\documentclass[11pt,a4paper,reqno]{amsart}
\usepackage[utf8]{inputenc}
\usepackage{amsmath,amsfonts,amscd,amssymb,latexsym,mathtools}
\usepackage[left=1.21in,right=1.2in,top=1in,bottom=1in]{geometry}

\usepackage{mathrsfs}
\usepackage[percent]{overpic} 

\usepackage[nobysame]{amsrefs}
\usepackage{enumitem}
\usepackage{comment}
\usepackage{soul}
\usepackage{tikz}
\usepackage{marginnote}
\usepackage{bbm}
\usepackage{cancel,slashed}
\usepackage[normalem]{ulem}

\usepackage{footmisc}
\usepackage{hyperref}
\hypersetup{
    colorlinks=true,
    linkcolor=blue,
    citecolor=green,}

\numberwithin{equation}{section}

\newtheorem{theorem}{Theorem}[section]
\newtheorem{lemma}[theorem]{Lemma}
\newtheorem{proposition}[theorem]{Proposition}
\newtheorem{corollary}[theorem]{Corollary}

\theoremstyle{definition}
\newtheorem{definition}[theorem]{Definition}
\newtheorem{convention}[theorem]{Convention}
\newtheorem{example}[theorem]{Example}

\theoremstyle{remark}
\newtheorem{remark}[theorem]{Remark}

\newcommand{\cc}{\mathbbm{c}}
\newcommand{\CC}{\mathbb{C}}
\newcommand{\DD}{\mathbb{D}}

\newcommand{\GG}{\mathbb{G}}

\newcommand{\RR}{\mathbb{R}}
\renewcommand{\SS}{\mathbb{S}}
\newcommand{\ZZ}{\mathbb{Z}}

\newcommand{\calD}{\mathcal{D}}

\newcommand{\calG}{{\mathcal{G}}}

\newcommand{\calL}{\mathcal{L}}

\newcommand{\calU}{\mathcal{U}}

\newcommand{\rmc}{\mathrm{c}}

\newcommand{\rmS}{\mathrm{S}}

\newcommand{\ch}{\operatorname{ch}}

\newcommand{\coker}{\operatorname{coker}}
\newcommand{\diam}{\operatorname{diam}}
\newcommand{\dist}{\operatorname{dist}}
\newcommand{\dom}{\operatorname{dom}}
\newcommand{\End}{\operatorname{End}}
\newcommand{\grad}{\operatorname{grad}}

\newcommand{\id}{\operatorname{id}}

\newcommand{\ind}{\operatorname{ind}}
\newcommand{\Lip}{\operatorname{Lip}}
\newcommand{\Rad}{\operatorname{Rad}}

\newcommand{\rk}{\operatorname{rk}}

\newcommand{\Ric}{\operatorname{Ric}}

\newcommand{\spec}{\operatorname{spec}}

\newcommand{\Vol}{\operatorname{Vol}}

\newcommand{\upper}{\uppercase\expandafter}

\newcommand{\SigM}{{\Sigma M}}
\newcommand{\SigS}{{\Sigma\rmS}}

\newcommand{\tilcalG}{{\tilde\calG}}

\newcommand{\n}{\nabla}
\newcommand{\p}{\partial}
\newcommand{\pM}{{\p M}}

\newcommand{\hatD}{\hat D}

\newcommand{\scal}{{\rm scal}}

\renewcommand{\oe}{\bar{e}}
\newcommand{\oG}{\bar{G}}

\newcommand{\onu}{\bar{\nu}}

\begin{document}

\normalsize

\title[Dirac eigenvalue and the hyper-hemispherical radius]{Dirac eigenvalues and the hyper-hemispherical radius}

\author[Pengshuai Shi]{Pengshuai Shi}
\address{School of Mathematics and Statistics, Beijing Institute of Technology, Beijing 100081, People's Republic of China}
%\address{Beijing Key Laboratory on Mathematical Characterization, Analysis, and Applications of Complex Information, Beijing Institute of Technology, Beijing 102488, P. R. China}

\email{shipengshuai@bit.edu.cn, pengshuai.shi@gmail.com}

%\subjclass[2020]{Primary 58J28; Secondary 58J20, 53C27, 58J30, 58J32, 58J35}
%\subjclass[2010]{Primary 58J28; Secondary 53C25, 53C27, 58J20, 58J30}


\thanks{The author was partially supported by the NSFC (grant no. 12101042), the Fundamental Research Funds for the Central Universities and Beijing Institute of Technology Research Fund Program for Young Scholars.}

\begin{abstract}
On a connected compact Riemannian spin manifold with boundary, we prove an upper estimate for the smallest eigenvalue of the Dirac operator under the boundary condition associated with a chirality operator. This upper bound is expressed in terms of the hyper-hemispherical radius. Moreover, the upper bound is realized if and only if the manifold is a hemisphere. As an application, we get a scalar curvature rigidity result for the hemisphere in the case of Lipschitz maps whose differential has a prescribed trace-norm bound. Our estimate extends B\"ar's result for closed manifolds, where the upper bound is given in terms of hyperspherical radius.
\end{abstract}

\maketitle

%\setcounter{tocdepth}{1}
%\tableofcontents

%%%%%%%%%%%%%%%%%%%%%%
%%%%%%%%%%%%%%%%%%%%%%
\section{Introduction}\label{S:intro}

Dirac operator is an important object in mathematics and physics. Its spectrum encodes rich geometric information of the underlying space. Notable early results in this regard include the Weyl law and Lichnerowicz theorem \cite{Lich63}. On a connected closed Riemannian spin $n$-manifold $M$, the Lichnerowicz formula gives a lower bound for the eigenvalues of the Dirac operator $D$ in terms of the scalar curvature:
\[
\lambda(D)^2\ge\frac{1}{4}\inf_M\scal.
\]
This was later improved by Friedrich \cite{Friedrich80} to the following sharp estimate
\[
\lambda(D)^2\ge\frac{n}{4(n-1)}\inf_M\scal.
\]
From then on, numerous Dirac eigenvalue (lower or upper) estimates have been established. We refer interested readers to \cite{Ginoux09book}.

In a recent paper \cite{Baer26hyp-rad}, B\"ar proves a new upper estimate for the smallest Dirac eigenvalue (in absolute value)
\[
\lambda_1(D)^2\le\frac{n^2}{4\Rad_{\rmS^n}(M)^2}.
\]
Here $\Rad_{\rmS^n}(M)$ denotes the \emph{hyperspherical radius} of $M$. It is defined to be the supremum of all $R>0$ such that there exists a Lipschitz map $f$ from $M$ to the round unit $n$-sphere $\rmS^n$ with Lipschitz constant $\Lip(f)\le1/R$ and of nonzero degree. This estimate is rigid with the equality characterized by a round sphere of radius $\Rad_{\rmS^n}(M)$.

The notion of hyperspherical radius took shape in the early work of Gromov and Lawson \cite{GromovLawson80spin} on the study of positive scalar curvature using Dirac operators. It roughly represents the largest possible radius of a sphere such that the manifold can wrap around this sphere in a nontrivial, distance-nonincreasing way. There are several interesting implications of B\"ar's estimate. In particular, when combined with Friedrich's estimate, it yields Llarull's rigidity theorem \cite{Llarull98} for Lipschitz distance-nonincreasing maps.

In this paper, we investigate the Dirac eigenvalue estimation on manifolds with boundary. While lower estimates were studied extensively (some of them will be mentioned below), upper estimates are sparse in this case. Our main result is an extension of B\"ar's estimate to manifolds with boundary. To this end, we first need a new version of hyperspherical radius adapted to our setting. This is done by changing the mapping target from spheres to hemispheres, and the resulting quantity is called \emph{hyper-hemispherical radius}, and denoted by $\Rad_{\rmS_+^n}(M)$ (cf. Definition \ref{D:HHS rad}).

Secondly, in order to formulate a spectral problem for the Dirac operator on manifolds with boundary, a suitable boundary condition must be given. Typically, one can consider global (APS type) or local (chirality, MIT bag) boundary conditions. We choose to work on the chirality boundary condition (cf. Section \ref{S:chi bc}) because it is easy to handle and cooperates nicely with the Killing spinors on the hemisphere. In fact, under a chirality boundary condition, Hijazi--Montiel--Rold\'an \cite{HMR02} (see also \cite{HMZ01boundary}) obtained a rigid Friedrich-type Dirac eigenvalue lower estimate (cf. Theorem \ref{T:Dirac eig lower}). Also, the chirality boundary condition is conformally invariant and thus closely related to the Yamabe problem.

With these preparations, we have the following estimate.

\begin{theorem}\label{T:Dirac eig HHS}
Let $M$ be a connected compact Riemannian spin $n$-manifold ($n\ge2$) with nonempty boundary, and $D$ be the spin Dirac operator on $M$. Then under a chirality boundary condition, the smallest eigenvalue of $D$ (in absolute value) satisfies
\begin{equation}\label{E:Dirac eig HHS}
\lambda_1(D)^2\le\frac{n^2}{4\Rad_{\rmS_+^n}(M)^2}.
\end{equation}
Equality holds if and only if $M$ is isometric to the round $n$-hemisphere $\rmS_+^n(R)$ with radius $R=\Rad_{\rmS_+^n}(M)$.
\end{theorem}

Notice that we do not need any assumptions on the second fundamental form or the mean curvature of $\pM$ in this result.

\begin{remark}\label{R:Dirac eig HHS}
As discussed in Section \ref{SS:chi op}, a chirality always exists on an even-dimensional spin manifold (and up to sign is equivalent to the standard $\ZZ_2$-grading of the spinor bundle). But on odd-dimensional manifolds, one has to take two copies of the spinor bundle to get a chirality. In this case, the Dirac operator is also understood as two copies of the individual Dirac operator (but appearing in the off-diagonal blocks). Although a chirality is needed to define the boundary condition, the squared Dirac eigenvalues actually do not depend on the chirality. See Remark \ref{R:chi bc}.
\end{remark}

\begin{remark}\label{R:Dirac eig HHS-2}
In \cite{Raulot09conformal}, Raulot defined a spin conformal invariant using the smallest Dirac eigenvalue under a chirality boundary condition, and proved that this invariant attains its maximum in the case of hemispheres. However, no rigidity results can be obtained.
\end{remark}

We point out that estimate \eqref{E:Dirac eig HHS} can be strengthened as
\begin{equation}\label{E:Dirac eig HHS-tr}
\lambda_1(D)^2\le\frac{n^2}{4\Rad_{\rmS_+^n}^{\rm tr}(M)^2},
\end{equation}
where $\Rad_{\rmS_+^n}^{\rm tr}(M)$ is a trace-refined hyper-hemispherical radius, cf. Section \ref{SS:ref HHS}.
However, equality in \eqref{E:Dirac eig HHS-tr} allows non-round spaces as shown in Proposition \ref{P:non-hemisp eq}, thus the rigidity part no longer holds. Nevertheless, when combined with the estimate of Hijazi--Montiel--Rold\'an, one does get a scalar curvature rigidity theorem for the hemisphere.

\begin{corollary}\label{C:hemisp rigidity}
Let $M$ be a connected compact Riemannian spin $n$-manifold ($n\ge2$) with nonempty boundary. Suppose the scalar curvature satisfies $\scal\ge n(n-1)$ and the normalized mean curvature $H_\pM$ is nonnegative (with respect to the inner unit normal vector).
Let $f:M\to\rmS_+^n$ be a Lipschitz map of nonzero degree such that $f(\pM)\subset\p\rmS_+^n$. If the trace norm $|df|_1\le n$ almost everywhere, then $f$ is a Riemannian isometry.
\end{corollary}

\begin{remark}\label{R:hemisp rigidity}
When the map $f:M\to\rmS_+^n$ is assumed to be 1-Lipschitz or area-nonincreasing, such scalar curvature rigidity theorems have been proved previously in different generalities (cf. \cite{Lott21pams,CHS26lip,LeeTam26ajm,WangXie25tams,CHSS25} etc.). As a comparison, our result puts the assumption on the trace norm of $df$. The area-nonincreasing and the trace norm assumptions are mutually incomparable in general, but they both include the case of 1-Lipschitz maps.
\end{remark}

When $n\ge3$, the estimate of Hijazi--Montiel--Rold\'an was improved by Raulot \cite{Raulot06} in terms of the first eigenvalue of the conformal Laplacian. The estimate reads
\begin{equation}\label{E:Dirac eig con Lap}
\lambda(D)^2\ge\frac{n}{4(n-1)}\mu_1(L),
\end{equation}
where $\mu_1(L)$ is the first eigenvalue of the eigenvalue problem
\[
\left\{
\begin{aligned}
& Lu:=-\frac{4(n-1)}{n-2}\Delta u+\scal\cdot u=\mu_1(L)u \qquad \mbox{on }M,\\
& Bu:=-\frac{2}{n-2}\frac{\p u}{\p\nu}+H_\pM\cdot u=0 \qquad  \mbox{along }\pM.
\end{aligned}
\right.
\]
Here $\nu$ denotes the inward unit normal vector field along $\pM$. Moreover, the equality is realized if and only if $M$ is isometric to a hemisphere. Combining \eqref{E:Dirac eig con Lap} and \eqref{E:Dirac eig HHS-tr} yields an improvement of Corollary \ref{C:hemisp rigidity} as follows with essentially the same proof.

\begin{corollary}\label{C:hemisp rigidity-conf}
Let $M$ be a connected compact Riemannian spin $n$-manifold ($n\ge3$) with nonempty boundary. Suppose $\mu_1(L)\ge n(n-1)$. Let $f:M\to\rmS_+^n$ be a Lipschitz map of nonzero degree such that $f(\pM)\subset\p\rmS_+^n$. If the trace norm $|df|_1\le n$ almost everywhere, then $f$ is a Riemannian isometry.
\end{corollary}

The paper is organized as follows. In Section \ref{S:HHS rad}, we define hyper-hemispherical radius and compute it for the example of spherical caps. In Section \ref{S:chi bc}, we discuss the chirality boundary condition, which is a local boundary condition associated with a chirality operator. In particular, we need to take the parity of the dimension into consideration. In Section \ref{S:ind Dirac}, we study the Dirac operator on a spinor bundle with a Lipschitz twist and find a nontrivial harmonic twisted spinor by index computations. This spinor is intimately connected to the Dirac eigenvalue estimate, and is used to prove Theorem \ref{T:Dirac eig HHS} in Section \ref{S:pf thm cor}. We then turn to the trace-refined hyper-hemispherical radius, discuss the equality case of \eqref{E:Dirac eig HHS-tr} and prove Corollary \ref{C:hemisp rigidity}. Finally in Section \ref{S:ref est}, we compute lower bounds of the trace-refined hyper-hemispherical radius for geodesic annuli and warped-product spaces to get explicit upper Dirac eigenvalue estimates.

%%%%%%%%%%%%%%%%%%%%%%
\subsection{Notational convention}\label{SS:nota con}

To avoid confusions, we unify the notations that are used throughout.

\begin{convention}\label{Con:spinor bundle}
Let $M$ be a connected spin manifold of dimension $n$. We denote by $\SS M$ the (complex) spinor bundle of rank $2^{[\frac{n}{2}]}$ over $M$. The Clifford multiplication and spin connection on $\SS M$ are denoted by $\cc(\cdot)$ and $\n^{\SS M}$, respectively (cf. \cite[Chapter II]{LawMic89}).

We will mostly work on the \emph{graded} spinor bundle over $M$, denoted by $\SigM$. It is of rank $2^{[\frac{n+1}{2}]}$ and is defined to be

\begin{itemize}
\item $\SigM:=\SS M$ when $n$ is even;
\item $\SigM:=\SS M\oplus\SS M$ when $n$ is odd.
\end{itemize}

In both dimension parities, the Clifford multiplication and spin connection on $\SigM$ are denoted by $\rmc(\cdot)$ and $\n^\SigM$, respectively. (See \eqref{E:spinor bundle-odd} for the odd-dimensional setting.)

When $M=\rmS^n$ or $\rmS_+^n$, we will use $\bar{\cc}(\cdot)$ (or $\bar{\rmc}(\cdot)$) to denote the Clifford multiplications on the corresponding spinor bundles.
\end{convention}

We point out that the reason to consider graded spinor bundles is to make it possible to talk about chiralities. See Section \ref{SS:chi op} for details.

%%%%%%%%%%%%%%%%%%%%%%
%%%%%%%%%%%%%%%%%%%%%%
\section{Hyper-hemispherical radius}\label{S:HHS rad}

In this section, we extend the definition of hyperspherical radius to manifolds with boundary. In this situation, the target space is replaced by a unit hemisphere.

\begin{definition}\label{D:HHS rad}
Let $M$ be a connected compact orientable Riemannian $n$-manifold ($n\ge2$) with nonempty boundary. The \emph{hyper-hemispherical radius} of $M$, denoted by $\Rad_{\rmS_+^n}(M)$, is defined to be the supremum of all $R>0$ such that there exists a boundary-preserving Lipschitz map $f:M\to\rmS_+^n$ with Lipschitz constant $\Lip(f)\le1/R$ and $\deg(f)\ne0$. Here $\rmS_+^n$ denotes the round unit (upper) $n$-hemisphere, and boundary-preserving means $f(\pM)\subset\p\rmS_+^n$.
\end{definition}

\begin{remark}\label{R:bdry pres}
The boundary-preserving condition is essential for the well-definedness of the degree. In particular, denoting $\p f:=f|_{\pM}:\pM\to\p\rmS_+^n$, then $\deg(f)=\deg(\p f)$.
\end{remark}

\begin{remark}\label{R:HHS rad}
The properties of hyperspherical radius discussed in \cite[Section 2]{Baer26hyp-rad} still hold for hyper-hemispherical radius. We record some of them below:
\begin{enumerate}
\item The supremum in Definition \ref{D:HHS rad} is a maximum. That is, there exists a boundary-preserving Lipschitz map $f:M\to\rmS_+^n$ with nonzero degree such that $\Lip(f)=1/\Rad_{\rmS_+^n}(M)$; \label{R:HHS rad-1}

\item It is equivalent to consider only \emph{smooth} maps $f$ in Definition \ref{D:HHS rad}, in which case the supremum may not be attained; \label{R:HHS rad-2}

\item The diameter of $M$ gives an upper bound for the hyper-hemispherical radius
\[
\Rad_{\rmS_+^n}(M)\le\frac{\diam(M)}{\pi};
\] \label{R:HHS rad-3}

\item It holds that
\[
\frac{|\deg(f)|}{\Lip(f)^n}\le\frac{\Vol(M)}{\Vol(\rmS_+^n)},
\]
which implies $\Rad_{\rmS_+^n}(M)^n\le\Vol(M)/\Vol(\rmS_+^n)$, with equality if and only if $M$ is isometric to the round hemisphere of radius $\Rad_{\rmS_+^n}(M)$; \label{R:HHS rad-4}

\item For any constant $\lambda>0$, we have $\Rad_{\rmS_+^n}(M,\lambda^2g)=\lambda\cdot\Rad_{\rmS_+^n}(M,g)$. \label{R:HHS rad-5}
\end{enumerate}
\end{remark}

\begin{example}\label{Eg:sph cap}
Consider the spherical cap
\[
M=\{x\in\rmS^n:\dist(x,p_0)\le\rho\}
\]
lying in the unit $n$-sphere (with the induced metric), where $p_0$ is the north pole and $0<\rho<\pi$.

In the case that $0<\rho\le\pi/2$, by Remark \ref{R:HHS rad}.\ref{R:HHS rad-3},
\begin{equation}\label{E:sph cap-1}
\Rad_{\rmS_+^n}(M)\le\frac{2\rho}{\pi}.
\end{equation}
In the case that $\pi/2<\rho<\pi$, note that $\pM$ is an $(n-1)$-sphere of radius $\sin\rho$, denoted by $\rmS^{n-1}(\sin\rho)$. For a boundary-preserving Lipschitz map $f:M\to\rmS_+^n$ of nonzero degree, the boundary map $\p f:\rmS^{n-1}(\sin\rho)\to\rmS^{n-1}$ is again of nonzero degree with the Lipschitz constant $\Lip(\p f)\le\Lip(f)$. Applying Remark \ref{R:HHS rad}.\ref{R:HHS rad-4} to $\p f$ yields $\Lip(\p f)\ge1/\sin\rho$. Hence $\Lip(f)\ge1/\sin\rho$. This means
\begin{equation}\label{E:sph cap-2}
\Rad_{\rmS_+^n}(M)\le\sin\rho.
\end{equation}

On the other hand, expressing each point on the $n$-sphere in geodesic polar coordinates $(\theta,\omega)$ with $0\le\theta\le\pi$ and $\omega\in\rmS^{n-1}$ (so the Cartesian coordinates become $(\omega\sin\theta,\cos\theta)$), we define a map $f:M\to\rmS_+^n$ by
\begin{equation}\label{E:sph cap map}
f(\theta,\omega):=\Big(\frac{\pi}{2\rho}\theta,\omega\Big).
\end{equation}
It is clear that $f$ is boundary-preserving. Also, $f$ is a homeomorphism so with degree $\pm1$. For a tangent vector $u=\p/\p\theta+v$ on $M$ (where $v$ is a tangent vector on $\rmS^{n-1}$), we have
\[
df(u)=\frac{\pi}{2\rho}\frac{\p}{\p\theta}+v.
\]
Using the standard metric $g=d\theta^2+\sin^2\theta g_{\rmS^{n-1}}$ on both $M$ and $\rmS_+^n$, we get
\[
\frac{|df(u)|_{\rmS_+^n}}{|u|_M}\le\max\bigg\{\frac{\pi}{2\rho},\sup_{0<\theta\le\rho}\frac{\sin(\frac{\pi}{2\rho}\theta)}{\sin\theta}\bigg\}.
\]
By a straightforward computation,
\[
\sup_{0<\theta\le\rho}\frac{\sin(\frac{\pi}{2\rho}\theta)}{\sin\theta}=
\begin{cases}
\frac{\pi}{2\rho}, & 0<\rho\le\frac{\pi}{2}, \\
\frac{1}{\sin\rho}, & \frac{\pi}{2}<\rho<\pi.
\end{cases}
\]
It follows that
\[
\Lip(f)\le
\begin{cases}
\frac{\pi}{2\rho}, & 0<\rho\le\frac{\pi}{2}, \\
\frac{1}{\sin\rho}, & \frac{\pi}{2}<\rho<\pi.
\end{cases}
\]
This implies the reverse inequality in \eqref{E:sph cap-1} or \eqref{E:sph cap-2}. In summary, we obtain
\begin{equation}\label{E:sph cap rad}
\Rad_{\rmS_+^n}(M)=\begin{cases}
\frac{2\rho}{\pi}, & 0<\rho\le\frac{\pi}{2}, \\
\sin\rho, & \frac{\pi}{2}<\rho<\pi.
\end{cases}
\end{equation}
In particular, $\Rad_{\rmS_+^n}(\rmS_+^n)=1$.

By Remark \ref{R:HHS rad}.\ref{R:HHS rad-5}, we also get the hyper-hemispherical radius for spherical caps lying in a round $n$-sphere of any radius. In particular,
\begin{equation}\label{E:HHS R-hemisp}
\Rad_{\rmS_+^n}(\rmS_+^n(R))=R.
\end{equation}
\end{example}

\begin{remark}\label{R:HHS rad example}
This example shows that the hyper-hemispherical radius of a spherical cap is determined by different reasons in the two regimes. For $0<\rho\le\pi/2$, it is determined by the bulk geometry; while for $\pi/2<\rho<\pi$, it is determined by the boundary geometry.
\end{remark}

\begin{remark}\label{R:HHS rad example-1}
Let $M$ be a geodesic ball of radius $\rho$ in the Euclidean or hyperbolic space. By a similar computation, one finds that $\Rad_{\rmS_+^n}(M)=2\rho/\pi$.
\end{remark}

%%%%%%%%%%%%%%%%%%%%%%
%%%%%%%%%%%%%%%%%%%%%%
\section{The chirality boundary condition on spin manifolds}\label{S:chi bc}

%%%%%%%%%%%%%%%%%%%%%%
\subsection{The chirality operator}\label{SS:chi op}

Let $M$ be a complete orientable Riemannian spin $n$-manifold and $\SigM\to M$ be the graded spinor bundle as in Convention \ref{Con:spinor bundle}. An endomorphism field $G$ of $\SigM$ is called a \emph{chirality operator} if it is involutive, unitary, parallel and anticommutes with the Clifford multiplication (cf. \cite[Section 1.5]{Ginoux09book}). A chirality operator $G$ determines a $\ZZ_2$-graded decomposition of the bundle into $\pm1$-eigenspaces of it.

For even-dimensional $M$, a chirality operator is naturally given by the Clifford multiplication of the complex volume form $\omega_\CC$, which under a positively-oriented local orthonormal tangent frame $(e_i)_{i=1,\dots,n}$, takes the form
\[
\omega_\CC:=(\sqrt{-1})^{[\frac{n+1}{2}]}e_1\wedge\dots\wedge e_n.
\]
And the corresponding grading is the standard $\ZZ_2$-grading on $\SS M$. In fact, the complex spinor representation is the irreducible module for the complex Clifford algebra $\CC l_n$ when $n$ is even, so $\pm\rmc(\omega_\CC)$ are the only possible chirality operators by Schur's lemma.

For odd-dimensional $M$, the complex volume form is central and acts as identity on $\SS M$. By a similar representation theory reasoning, there does not exist any chirality operator on $\SS M$ alone. But now we are on $\SigM$, which consists of two copies of $\SS M$. Chirality operators can be constructed in this situation. Basically we just define the spin connection and Clifford multiplication by
\begin{equation}\label{E:spinor bundle-odd}
\n^{\SigM}:=\left(
\begin{matrix}
\n^{\SS M} & 0 \\
0 & \n^{\SS M}
\end{matrix}
\right),\qquad
\rmc(\cdot):=\left(
\begin{matrix}
0 & \cc(\cdot) \\
\cc(\cdot) & 0
\end{matrix}
\right).
\end{equation}
Then
\begin{equation}\label{E:chi op-odd}
G_\theta=\left(
\begin{matrix}
\cos\theta & -\sqrt{-1}\sin\theta \\
\sqrt{-1}\sin\theta & -\cos\theta
\end{matrix}
\right),\qquad\theta\in[0,2\pi]
\end{equation}
is a family of chirality operators on $\SigM$.

\begin{remark}\label{R:chi op-odd}
The above $G_\theta$ are actually all the possible chirality operators on $\SigM$. To see this, write
\[
G=\left(
\begin{matrix}
T_1 & T_2 \\
T_3 & T_4
\end{matrix}
\right),\qquad T_i\in\End(\SS M).
\]
Since $G$ anticommutes with $\rmc(\cdot)$, we get
\[
\begin{aligned}
& T_1\cc(\cdot)+\cc(\cdot)T_4=T_4\cc(\cdot)+\cc(\cdot)T_1=0, \\
& T_2\cc(\cdot)+\cc(\cdot)T_3=T_3\cc(\cdot)+\cc(\cdot)T_2=0.
\end{aligned}
\]
It then follows that for $i\in\{1,2,3,4\}$,
\[
T_i\cc(X)\cc(Y)=\cc(X)\cc(Y)T_i,\qquad\mbox{for all }X,Y\in TM.
\]
That is, $T_i$ commutes with the action of the even Clifford algebra $\CC l_n^0\cong\CC l_{n-1}$. Note that $n-1$ is even, again by Schur's lemma, $T_i=t_i\id_{\SS M}$ for scalar functions $t_i:M\to\CC$. Since $G$ is parallel, each $t_i$ is constant on the connected $M$. Moreover, $t_1+t_4=t_2+t_3=0$. The expression \eqref{E:chi op-odd} now follows from the fact that $G$ is a unitary involution.

Furthermore, for $\theta\in[0,2\pi]$, it holds that $G_\theta=\calU_\theta G_0\calU_\theta^{-1}$ with
\[
\calU_\theta:=\left(
\begin{matrix}
\cos\frac{\theta}{2} & \sqrt{-1}\sin\frac{\theta}{2} \\
\sqrt{-1}\sin\frac{\theta}{2} & \cos\frac{\theta}{2}
\end{matrix}
\right).
\]
Also $\calU_\theta$ commutes with both $\nabla^{\SigM}$ and $\rmc(\cdot)$.
\end{remark}

\begin{remark}\label{R:chi op}
In the literature, it was mentioned that if $M$ is (or is a codimension zero submanifold of) a spacelike hypersurface of a Lorentzian manifold, then Clifford multiplication by a unit timelike normal vector field serves as a chirality operator. It is worth pointing out that this chirality should be regarded as on the restriction of the ambient spinor bundle instead of on the intrinsic spinor bundle $\SS M$, as the former is actually two copies of the latter in the odd-dimensional situation. Precisely, the chirality operator given by such a Clifford multiplication is $G_0$ or $G_\pi$ of \eqref{E:chi op-odd}.
\end{remark}

With a chirality operator chosen, $\SigM$ is $\ZZ_2$-graded
\[
\SigM=\Sigma^+M\oplus\Sigma^-M.
\]
The Dirac operator on $\SigM$ is defined as usual to be
\[
D:=\sum_{i=1}^n\rmc(e_i)\n_{e_i}^\SigM.
\]
It is formally self-adjoint, and with respect to the grading has the form
\[
D=\left(
\begin{matrix}
0 & D^- \\
D^+ & 0
\end{matrix}
\right),
\]
where for $n$ even, $D^\pm$ are the two chiral Dirac operators, and for $n$ odd, $D^\pm$ are both the formally self-adjoint Dirac operators on $\SS M$.

%%%%%%%%%%%%%%%%%%%%%%
\subsection{The chirality boundary condition}\label{SS:chi bc}

Consider $M$ to be a compact Riemannian spin $n$-manifold with nonempty boundary admitting a chirality operator $G$. Define an involution on $\SigM|_\pM$ by $\calG:=\rmc(\nu)G$, where $\nu$ is the inward unit normal vector field along $\pM$. Then one can impose a local boundary condition to the Dirac operator $D$, whose domain is given by
\begin{equation}\label{E:chi bc}
H^1_\calG(M,\SigM):=\{\varphi\in H^1(M,\SigM):\calG\varphi|_\pM=\varphi|_\pM\}.
\end{equation}
This is called the \emph{chirality boundary condition}. The resulting operator is denoted by $D_\calG$.
It can be shown that this is an elliptic boundary condition (cf. \cite{Bunke93comparison}) and $D_\calG$ is a self-adjoint operator. Hijazi--Montiel--Rold\'an proved the following eigenvalue estimate. (See also \cite{HMZ01boundary}.)

\begin{theorem}[{\cite[Theorem 3]{HMR02}}]\label{T:Dirac eig lower}
Suppose $\pM$ has nonnegative mean curvature (with respect to the inward normal). Then any (real) eigenvalue of $D$ under a chirality boundary condition satisfies a squared lower bound in terms of the scalar curvature of $M$:
\begin{equation}\label{E:eig lower}
\lambda^2\ge\frac{n}{4(n-1)}\inf_M\scal.
\end{equation}
Moreover, equality holds (for the smallest eigenvalues in absolute value) if and only if $M$ is isometric to the round $n$-hemisphere of radius $n/2|\lambda|$.
\end{theorem}

\begin{remark}\label{R:Dirac square eig}
Consider the self-adjoint operator $D^2$ with the domain
\[
\dom(D_\calG^2):=\{\varphi\in H^2(M,\SigM):\varphi\in H^1_\calG(M,\SigM),\,D\varphi\in H^1_\calG(M,\SigM)\}.
\]
Then $\spec(D_\calG^2)=\{\lambda^2:\lambda\in\spec(D_\calG)\}$. So the theorem is equivalent to a lower eigenvalue estimate for $D_\calG^2$.
\end{remark}

\begin{remark}\label{R:chi bc}
If $\calG$ is replaced by $-\calG$ in \eqref{E:chi bc}, the resulting boundary condition is called \emph{negative} chirality boundary condition. (In this regard, the usual chirality boundary condition is sometimes called positive chirality boundary condition.) Also, for $\varphi\in H^1_\calG(M,\SigM)$, one can see that $G\varphi\in H^1_{-\calG}(M,\SigM)$. Since $D$ anticommutes with $G$, the Dirac eigenvalues under positive and negative chirality boundary conditions are equal in absolute value. Thus we do not need to distinguish them in the square Dirac eigenvalue estimate.

In fact, one can further show that the (squared) Dirac eigenvalues are independent of the chirality operators. In even dimensions, we already see it from above as the only chirality operators are $\pm\rmc(\omega_\CC)$. In odd dimensions, by the discussion in Remark \ref{R:chi op-odd}, $\calU_\theta$ induces an isomorphism between $H_{\calG_0}^1(M,\SigM)$ and $H_{\calG_\theta}^1(M,\SigM)$. Since $\calU_\theta$ commutes with $D$, the eigenvalues are preserved.
\end{remark}

From this remark, we can make the following convention from now on.

\begin{convention}\label{Con:chi op}
When $n=\dim M$ is even, the chirality operator is chosen to be $G=\rmc(\omega_\CC)$.
When $n=\dim M$ is odd, the chirality operator is chosen to be $G=G_0$ in \eqref{E:chi op-odd}.
%(Note that this is equivalent to setting $\cc(\omega_\CC)=\id$ on the first copy of $\SS M$ and $\cc(\omega_\CC)=-\id$ on the second copy.)
\end{convention}

%%%%%%%%%%%%%%%%%%%%%%
\subsection{Killing spinors on the hemisphere}\label{SS:Kil spi hemisp}

It is well-known (see for example \cite[Appendix A]{Ginoux09book}) that there exists a family of smooth spinor fields $(\psi'_\alpha)_{\alpha=1,\dots,2^{[\frac{n}{2}]}}$ on $\rmS^n$ such that
\begin{equation}\label{E:Killing spinor'}
\n^{\SS\rmS^n}_\cdot\psi'_\alpha=-\frac{1}{2}\bar{\cc}(\cdot)\psi'_\alpha,
\end{equation}
where $\bar{\cc}(\cdot)$ denotes the Clifford multiplication on $\SS\rmS^n$. These so-called \emph{Killing spinors} form an orthonormal basis of the spinor bundle $\SS\rmS^n$ at each point of $\rmS^n$. Also, there is another family of smooth spinor fields $(\psi''_\alpha)_{\alpha=1,\dots,2^{[\frac{n}{2}]}}$ satisfying the above properties except that $-1/2$ is replaced by $1/2$ in \eqref{E:Killing spinor'}.

For the graded spinor bundle $\SigS^n_+$ over the hemisphere, we can choose such a frame adapted to the chirality boundary conditions.

\begin{lemma}\label{L:Killing spinor}
There exists a smooth orthonormal frame $(\psi_\alpha)_{\alpha=1,\dots,2^{[\frac{n+1}{2}]}}$ of the graded spinor bundle $\SigS^n_+$ which satisfies the Killing equations
\begin{equation}\label{E:Killing spinor}
\n^{\SigS^n_+}_\cdot\psi_\alpha=-\frac{1}{2}\bar{\rmc}(\cdot)\psi_\alpha.
\end{equation}
Moreover, we can require that the first half of these spinors satisfies the positive chirality boundary condition associated with the chirality operator $\oG$ and the second half satisfies the negative chirality boundary condition.
\end{lemma}

\begin{proof}
Note that the involution $\bar{\calG}=\bar{\rmc}(\onu)\oG$ induces a splitting of $\SigS_+^n|_{\p\rmS_+^n}$ into two eigenspaces of the same rank. Choose an orthonormal basis at a boundary point with the first half of the elements lying in the $+1$-eigenspace and the second half lying in the $-1$-eigenspace. Then extend using the parallel transport of the Killing connection
\[
\n^{\SigS^n_+,\frac{1}{2}}_\cdot:=\n^{\Sigma\rmS^n_+}_\cdot+\frac{1}{2}\bar{\rmc}(\cdot)
\]
to get a global family of spinor fields $(\psi_\alpha)_\alpha$. The well-definedness is guaranteed by the fact that $\rmS^n_+$ is simply connected and by the curvature identity
\[
R^{\SigS^n_+,\frac{1}{2}}_{X,Y}=R^{\SigS^n_+}_{X,Y}+\frac{1}{4}[\bar{\rmc}(X),\bar{\rmc}(Y)]=0
\]
for any tangent vectors $X,Y$, where the latter equality is a well-known formula on the unit round sphere (with no differences for the graded spinor bundle).
Now it is clear that each $\psi_\alpha$ satisfies \eqref{E:Killing spinor}. The orthonormal property follows from
\[
\begin{aligned}
X\langle\psi_\alpha,\psi_\beta\rangle & =\langle\n^{\Sigma\rmS^n_+}_X\psi_\alpha,\psi_\beta\rangle+\langle\psi_\alpha,\n^{\Sigma\rmS^n_+}_X\psi_\beta\rangle \\
& =-\frac{1}{2}\langle\bar{\rmc}(X)\psi_\alpha,\psi_\beta\rangle-\frac{1}{2}\langle\psi_\alpha,\bar{\rmc}(X)\psi_\beta\rangle=0,
\end{aligned}
\]
where $\langle\cdot,\cdot\rangle$ denotes the pointwise Hermitian inner product. Finally, the boundary condition is preserved because $\bar{\calG}$ is $\n^{\SigS^n_+,\frac{1}{2}}$-parallel along the boundary (the equator).
\end{proof}

The same argument applies with $1/2$ replaced by $-1/2$ in \eqref{E:Killing spinor}, using the connection $\n^{\Sigma\rmS^n_+}_\cdot-\frac{1}{2}\bar{\rmc}(\cdot)$.

%%%%%%%%%%%%%%%%%%%%%%
%%%%%%%%%%%%%%%%%%%%%%
\section{Index computation for the twisted Dirac operator}\label{S:ind Dirac}

In this section, we work on the first step in proving Theorem \ref{T:Dirac eig HHS} --- finding a nontrivial element in the kernel of the twisted Dirac operator on $M$. In doing this, we need a suitable boundary condition and an index theoretic argument.

%%%%%%%%%%%%%%%%%%%%%%
\subsection{Spinor bundle with Lipschitz twists}\label{SS:Lip twi}

Let $M$ be as in Section \ref{SS:chi bc} and $f:M\to\rmS_+^n$ be a boundary-preserving Lipschitz map of nonzero degree. Form two twisted spinor bundles
\[
E:=\SigM\otimes f^*\SigS_+^n\quad\mbox{ and }\quad E_0:=\SigM\otimes\CC(M)^{2^{[\frac{n+1}{2}]}}
\]
over $M$, where $\CC(M)^{2^{[\frac{n+1}{2}]}}\to M$ denotes the trivial complex bundle of rank $2^{[\frac{n+1}{2}]}=\rk(f^*\SigS_+^n)$. The pull-back $f^*\SigS_+^n$ is a Lipschitz bundle. Nevertheless, the $H^1$-space $H^1(M,E)$ and $L^2$-space $L^2(M,E)$ can still be defined from completions of Lipschitz sections of $E$. We then have a twisted Dirac operator
\[
D_f:=\sum_{i=1}^n\rmc(e_i)\n^f_{e_i}:H^1(M,E)\to L^2(M,E),
\]
where $\rmc(\cdot)$ denotes the Clifford multiplication on $\SigM$, $\n^f:=\n^\SigM\otimes1+1\otimes f^*\n^{\SigS_+^n}$ is the tensor product connection on $E$, and $(e_i)_{i=1,\dots,n}$ is a local orthonormal frame of $TM$. Similarly, there is another twisted Dirac operator
\[
D_0:H^1(M,E_0)\to L^2(M,E_0).
\]
It is actually $2^{[\frac{n+1}{2}]}$ copies of $D:H^1(M,\SigM)\to L^2(M,\SigM)$. So we will denote both of them by $D$ henceforth.

We now recall the definition of a unitary map between $E$ and $E_0$ (cf. \cite[Section 4]{Baer26hyp-rad}). Using the family of spinor fields $(\psi_\alpha)_{\alpha=1,\dots,2^{[\frac{n+1}{2}]}}$ constructed in Lemma \ref{L:Killing spinor}, we get a Lipschitz continuous family of sections $\Psi_\alpha:=f^*\psi_\alpha$ that form an orthonormal basis of $f^*\SigS_+^n$ at every point of $M$. Each section $\varphi$ of $E$ can be expressed as
\begin{equation}\label{E:phi decomp}
\varphi=\sum_\alpha\varphi_\alpha\otimes\Psi_\alpha.
\end{equation}
This induces a unitary bundle isomorphism
\begin{equation*}
\begin{aligned}
U_f:E &\to E_0, \\
\varphi &\mapsto(\varphi_\alpha)_{\alpha=1,\dots,2^{[\frac{n+1}{2}]}}
\end{aligned}
\end{equation*}
of Lipschitz regularity. In this case, $\varphi\in H^1(M,E)$ if and only if every $\varphi_\alpha\in H^1(M,\Sigma M)$. The map $U_f$ builds a bridge between $D_f$ and $D$. In fact, the following lemma shows that they just differ by an $L^\infty$-potential.

\begin{lemma}[{\cite[Lemma 1]{Baer26hyp-rad}}]\label{L:D_f vs. D}
It holds that
\[
U_f\circ D_f\circ U_f^{-1}=D-\frac{1}{2}Z_f,
\]
where $Z_f$ is an $L^\infty$-endomorphism field on $E_0$ defined by
\[
Z_f(\varphi_\alpha)_\alpha=\Big(\sum_\beta\sum_{i=1}^n\langle df(e_i)\psi_\beta,\psi_\alpha\rangle\rmc(e_i)\varphi_\beta\Big)_\alpha.
\]
\end{lemma}

%%%%%%%%%%%%%%%%%%%%%%
\subsection{Boundary conditions for $D_f$}\label{SS:D_f bc}

%We discuss the boundary condition for $D_f$ so as to get a well-posed index problem. 
Having the chirality operators $G$ on $M$ and $\oG$ on $\rmS_+^n$ as in Convention \ref{Con:chi op}, respectively, the corresponding spinor bundles are $\ZZ_2$-graded
\begin{equation}\label{E:SigM decom}
\SigM=\Sigma^+M\oplus\Sigma^-M,\qquad\SigS_+^n=\Sigma^+\rmS_+^n\oplus\Sigma^-\rmS_+^n.
\end{equation}
This induces a total $\ZZ_2$-grading on $E$ as
\begin{equation}\label{E:E decom}
\begin{aligned}
& E^+:=\big(\Sigma^+M\otimes f^*\Sigma^+\rmS_+^n\big)\oplus\big(\Sigma^-M\otimes f^*\Sigma^-\rmS_+^n\big), \\
& E^-:=\big(\Sigma^+M\otimes f^*\Sigma^-\rmS_+^n\big)\oplus\big(\Sigma^-M\otimes f^*\Sigma^+\rmS_+^n\big).
\end{aligned}
\end{equation}
With respect to this grading, the Dirac operator is odd graded
\[
D_f=\left(
\begin{matrix}
0 & D_f^- \\
D_f^+ & 0
\end{matrix}
\right),
\]
where $D_f^\pm$ are formal adjoints of each other.

Using $G$ and $\oG$ again, we define an involution on $E|_{\pM}$ by
\[
\tilcalG:=\rmc(\nu)G\otimes f^*(\bar{\rmc}(\onu)\oG).
\]
Similar to Section \ref{SS:chi bc}, we impose a chirality boundary condition for $D_f$ with domain $H_\tilcalG^1(M,E)$ defined as in \eqref{E:chi bc}. A priori, this is not a smooth operator. But like Lemma \ref{L:D_f vs. D}, one can compute that
\begin{equation}\label{E:bc conju}
U_f\circ\tilcalG\circ U_f^{-1}(\varphi_\alpha)_\alpha=\Big(\sum_\beta\langle\bar{\rmc}(\onu)\oG\psi_\beta,\psi_\alpha\rangle\rmc(\nu)G\varphi_\beta\Big)_\alpha,
\end{equation}
which is independent of $f$. Hence the theory of boundary value problems for Dirac-type operators with locally $L^\infty$-potentials due to Ballmann--Br\"uning--Carron \cite{BBC08} can be applied. In particular, $D_{f,\tilcalG}:H_\tilcalG^1(M,E)\to L^2(M,E)$ is self-adjoint and Fredholm.

\begin{remark}\label{R:D_f bc}
Note that $\tilcalG$ preserves $E^\pm$. When $n=\dim M$ is odd, it is more convenient to interpret the above setting in the ungraded language. By Conventions \ref{Con:spinor bundle} and \ref{Con:chi op}, $\Sigma^+M=\Sigma^-M=\SS M$, $\Sigma^+\rmS_+^n=\Sigma^-\rmS_+^n=\SS\rmS_+^n$, so that
\[
E^+=E^-=(\SS M\otimes f^*\SS\rmS_+^n)\oplus(\SS M\otimes f^*\SS\rmS_+^n).
\]
In this case
\[
D_f^+=\left(
\begin{matrix}
0 & \DD_f \\
\DD_f & 0
\end{matrix}
\right),
\]
where $\DD_f$ denotes the formally self-adjoint twisted Dirac operator acting on $\SS M\otimes f^*\SS\rmS_+^n$. Similarly for $D_f^-$.

For the boundary involution $\tilcalG$, with respect to the gradings \eqref{E:SigM decom},
\[
\rmc(\nu)G=\left(
\begin{matrix}
0 & -\cc(\nu) \\
\cc(\nu) & 0
\end{matrix}
\right),\qquad\bar{\rmc}(\onu)\oG=\left(\begin{matrix}
0 & -\bar{\cc}(\onu) \\
\bar{\cc}(\onu) & 0
\end{matrix}
\right).
\]
Hence
\[
\tilcalG|_{E^\pm}=\left(
\begin{matrix}
0 & \pm\cc(\nu)\otimes f^*\bar{\cc}(\onu) \\
\pm\cc(\nu)\otimes f^*\bar{\cc}(\onu) & 0
\end{matrix}
\right).
\]
%Thus with respect to the total grading \eqref{E:E decom},
%\[
%\tilcalG=\left(
%\begin{matrix}
%\left(
%\begin{matrix}
%0 & \cc(\nu)\otimes f^*\bar{\cc}(\onu) \\
%\cc(\nu)\otimes f^*\bar{\cc}(\onu) & 0
%\end{matrix}
%\right) & 0 \\
%0 & \left(
%\begin{matrix}
%0 & -\cc(\nu)\otimes f^*\bar{\cc}(\onu) \\
%-\cc(\nu)\otimes f^*\bar{\cc}(\onu) & 0
%\end{matrix}
%\right)
%\end{matrix}
%\right).
%\]
Set $\GG:=\cc(\nu)\otimes f^*\bar{\cc}(\onu)$ and
\[
H^1_{\GG}(M,\SS M\otimes f^*\SS\rmS_+^n):=\{\varphi\in H^1(M,\SS M\otimes f^*\SS\rmS_+^n):\GG\varphi|_\pM=\varphi|_\pM\}.
\]
Consider
\[
\DD_{f,\pm\GG}:H^1_{\pm\GG}(M,\SS M\otimes f^*\SS\rmS_+^n)\to L^2(M,\SS M\otimes f^*\SS\rmS_+^n).
\]
If there is a nonzero $\eta\in\ker\DD_{f,\pm\GG}$, put $\varphi=(\eta,\pm\eta)\in E^+$. Then $D_f^+\varphi=0$ and $\tilcalG\varphi=\varphi$. So $\varphi$ is a nonzero element of $\ker D_{f,\tilcalG}^+$.

\end{remark}

%%%%%%%%%%%%%%%%%%%%%%
\subsection{The index computation}\label{SS:ind D_f}

By computing the index of $D_{f,\tilcalG}^+$ or $\DD_{f,\GG}$, we aim to prove the following.

\begin{proposition}\label{P:ind D_f}
There exists a nonzero element in the kernel of $D_{f,\tilcalG}:H_\tilcalG^1(M,E)\to L^2(M,E)$.
\end{proposition}

\begin{proof}
For any Lipschitz map $f:M\to\rmS_+^n$, one can always find a smooth map which is $C^0$-close to $f$. This map, a priori, may not be boundary-preserving. But after a modification in a collar neighborhood of the boundary, we can eventually get a smooth boundary-preserving map $f_1:M\to\rmS_+^n$ with $\dist(f(x),f_1(x))<\pi$ for each $x\in M$. This gives a homotopy $f_t,t\in[0,1]$ of boundary-preserving maps between $f$ and $f_1$ by connecting $f(x)$ and $f_1(x)$ through the unique shortest geodesic in $\rmS_+^n$. This preserves the degree. Now spin parallel transport from $f_t(x)$ to $f_1(x)$ induces a Lipschitz unitary identification of $f_t^*\SigS_+^n$ and $f_1^*\SigS_+^n$. It preserves the $\ZZ_2$-grading and the boundary involution as $\oG$ is parallel and $\onu$ is parallel along the equator. One then gets a Lipschitz unitary identification of twisted spinor bundles $E_{f_t}$ and $E_{f_1}$ over $M$, again preserving the total $\ZZ_2$-grading and the boundary condition. Under this identification, $D_{f_t}^+$ and $D_{f_1}^+$ differ only by a zeroth-order $L^\infty$-potential which depends continuously on $t$. Thus $D_{f_t,\tilcalG_t}^+$ is a norm-continuous family of Fredholm operators, so the index is invariant. In odd dimensions, the invariance of $\ind\DD_{f_t,\GG_t}$ can be shown in a similar way. Therefore we can assume that $f$ is itself a \emph{smooth} map without loss of generality.

The index problem will be analysed according to the parity of $n=\dim M$ separately.

(\romannumeral1) Let $n$ be even. We compute the index of $D_{f,\tilcalG}^+$.

For $\varphi\in H_\tilcalG^1(M,E)$, it satisfies $\tilcalG\varphi|_\pM=\varphi|_\pM$. Write
\[
\varphi|_\pM=(u^{++},u^{--},u^{+-},u^{-+})
\]
with respect to the decomposition \eqref{E:E decom}. Since $\rmc(\cdot)$ (resp. $\bar{\rmc}(\cdot)$) interchanges $\Sigma^\pm M$ (resp. $\Sigma^\pm\rmS_+^n$), we have
\[
\begin{aligned}
& \big(\rmc(\nu)\otimes f^*\bar{\rmc}(\onu)\big)u^{++}=u^{--},\qquad \big(\rmc(\nu)\otimes f^*\bar{\rmc}(\onu)\big)u^{+-}=-u^{-+}, \\
& \big(\rmc(\nu)\otimes f^*\bar{\rmc}(\onu)\big)u^{--}=u^{++},\qquad \big(\rmc(\nu)\otimes f^*\bar{\rmc}(\onu)\big)u^{-+}=-u^{+-}.
\end{aligned}
\]
Actually the two lines are equivalent. This is the same as the boundary condition considered by Lott in \cite[Section 2.2]{Lott21pams}. (Note that \cite{Lott21pams} uses a different convention for the Clifford multiplication, but is equivalent to our setting.) %Subject to our setting, Clifford multiplication by the unit normal vector field is multiplied by $\pm1$ while changing chiralities.) 
Hence we can apply the index formula in \cite[Section 2.2]{Lott21pams} to get
\[
\ind D_{f,\tilcalG}^+=\int_M\hat{A}(M)f^*(\ch(\Sigma^+\rmS_+^n)-\ch(\Sigma^-\rmS_+^n))=(-1)^{\frac{n}{2}}\deg(f)\chi(\rmS_+^n)\ne0.
\]
Therefore, there must exist a nonzero element in $\ker D_{f,\tilcalG}^+$ or $\coker D_{f,\tilcalG}^+\cong\ker D_{f,\tilcalG}^-$. The proposition follows.

(\romannumeral2) Let $n$ be odd. In view of the interpretation in Remark \ref{R:D_f bc}, it suffices to compute the index of $\DD_{f,\GG}$.

Note that the restriction $\Sigma^+M|_\pM=\SS M|_\pM$ can be identified with the spinor bundle $\Sigma\pM=\SS\pM$ over the even-dimensional $\pM$ (with the induced spin structure), whose Clifford multiplication is $\rmc_\pM(\cdot)=\cc(\cdot)\cc(\nu)$. Moreover, the $\pm\sqrt{-1}$-eigenspaces of $\cc(\nu)$ give rise to the $\ZZ_2$-grading $\Sigma\pM=\Sigma^+\pM\oplus\Sigma^-\pM$. (See for example \cite[Section 2]{Baer96spinor} for details.)

From above, $\SS M\otimes f^*\SS\rmS_+^n|_\pM$ is a twisted spinor bundle $\Sigma\pM\otimes(\p f)^*\Sigma\rmS^{n-1}=:F$ with a total $\ZZ_2$-grading
\[
\begin{aligned}
& F^+:=\big(\Sigma^+\pM\otimes(\p f)^*\Sigma^+\rmS^{n-1}\big)\oplus\big(\Sigma^-\pM\otimes(\p f)^*\Sigma^-\rmS^{n-1}\big), \\
& F^-:=\big(\Sigma^+\pM\otimes(\p f)^*\Sigma^-\rmS^{n-1}\big)\oplus\big(\Sigma^-\pM\otimes(\p f)^*\Sigma^+\rmS^{n-1}\big).
\end{aligned}
\]
Also, $F^\pm$ corresponds exactly to the space of boundary values of $H_{\mp\GG}^1(M,\SS M\otimes f^*\SS\rmS_+^n)$.
Let $D_{\p f}$ denote the twisted Dirac operator on $F$. It is odd-graded with respect to the above grading. In addition, $D_{\p f}$ is an adapted operator to $\DD_f$ in the sense of \cite{BaerBallmann12}. Therefore, we can apply the index formula of B\"ar--Brendle--Hanke--Wang \cite[Theorem B.1]{BBHW24} to get
\[
\ind\DD_{f,\GG}=-\frac{1}{2}\ind D_{\p f}^+,
\]
where $D_{\p f}^+:H^1(\pM,F^+)\to L^2(\pM,F^-)$ is the positive chiral component of $D_{\p f}$. Its index can be computed by the classical Atiyah--Singer index theorem
\[
\begin{aligned}
\ind D_{\p f}^+ & =\int_\pM\hat{A}(\pM)(\p f)^*(\ch(\Sigma^+\rmS^{n-1})-\ch(\Sigma^-\rmS^{n-1})) \\
& =(-1)^{\frac{n-1}{2}}\deg(\p f)\chi(\rmS^{n-1})\ne0.
\end{aligned}
\]
So either $\ker\DD_{f,\GG}$ or $\ker\DD_{f,-\GG}$ is nontrivial. This shows the existence of a nonzero element in $\ker D_{f,\tilcalG}^+$, and completes the proof.
\end{proof}

\begin{remark}\label{R:ind D_f}
The idea of the above proof can be compared with \cite[Section 5]{Sun24hyp} (see also \cite{CHZ24fill-in}). Unlike \cite{Baer26hyp-rad}, we need neither spectral flow nor taking product with $\rmS^1$ here.
\end{remark}

\begin{corollary}\label{C:ker D_f chi}
There exists a nonzero $\varphi\in H^1(M,E)$ such that $D_f\varphi=0$ and $U_f\varphi=(\varphi_\alpha)_\alpha$ fulfills $\varphi_\alpha\in H_\calG^1(M,\SigM)$ or $\varphi_\alpha\in H_{-\calG}^1(M,\SigM)$ for every $\alpha$. In other words, $\varphi_\alpha$ satisfies either the positive or negative chirality boundary condition determined by $\calG=\rmc(\nu)G$.
\end{corollary}

\begin{proof}
By Proposition \ref{P:ind D_f}, there is a nonzero $\varphi\in H_\tilcalG^1(M,E)$ such that $D_f\varphi=0$. Write
\[
\varphi=\sum_\alpha\varphi_\alpha\otimes f^*\psi_\alpha
\]
as in \eqref{E:phi decomp}. Then
\[
\begin{aligned}
\tilcalG(\varphi|_\pM) & =\sum_\alpha\rmc(\nu)G(\varphi_\alpha|_\pM)\otimes f^*(\bar{\rmc}(\onu)\oG(\psi_\alpha|_\pM)) \\
& =\sum_{\text{first half of }\alpha\text{'s}}\rmc(\nu)G(\varphi_\alpha|_\pM)\otimes f^*(\psi_\alpha|_\pM) \\
& \quad -\sum_{\text{second half of }\alpha\text{'s}}\rmc(\nu)G(\varphi_\alpha|_\pM)\otimes f^*(\psi_\alpha|_\pM).
\end{aligned}
\]
Here $(\psi_\alpha)_\alpha$ are chosen to satisfy the positive or negative chirality boundary condition associated with $\oG$, as asserted in Lemma \ref{L:Killing spinor}. On the other hand,
\[
\tilcalG(\varphi|_\pM)=\varphi|_\pM=\sum_\alpha\varphi_\alpha|_\pM\otimes f^*(\psi_\alpha|_\pM).
\]
These combined imply $\varphi_\alpha\in H_\calG^1(M,\SigM)$ for the first half of $\alpha$'s, and $\varphi_\alpha\in H_{-\calG}^1(M,\SigM)$ for the second half of $\alpha$'s. The thesis then follows.
\end{proof}

%%%%%%%%%%%%%%%%%%%%%%
%%%%%%%%%%%%%%%%%%%%%%
\section{Proofs of Theorem \ref{T:Dirac eig HHS} and Corollary \ref{C:hemisp rigidity}}\label{S:pf thm cor}

In this section, we finish the proof of Theorem \ref{T:Dirac eig HHS}. As in \cite{Baer26hyp-rad}, the idea is to derive the desired eigenvalue estimate of $D$ from a section in the kernel of $D_f$. We then discuss the trace-refined hyper-hemispherical radius and present the proof of Corollary \ref{C:hemisp rigidity}.

%%%%%%%%%%%%%%%%%%%%%%
\subsection{Estimate of $D$ on the kernel of $D_f$}\label{SS:D est kerDf}

In this subsection, we provide some pointwise estimates for the action of $D$ on the kernel elements of $D_f$. This does not involve the boundary condition.

Let $\varphi=\sum_\alpha\varphi_\alpha\otimes\Psi_\alpha\in H^1(M,E)$ and $(\varphi_\alpha)_\alpha=U_f(\varphi)\in H^1(M,E_0)$ be as in last section. Then
\begin{equation}\label{E:Df and D}
D_f\varphi=\sum_\alpha D\varphi_\alpha\otimes\Psi_\alpha+\sum_{\alpha,i}\rmc(e_i)\varphi_\alpha\otimes (f^*\nabla^{\SigS_+^n})_{e_i}\Psi_\alpha,
\end{equation}
where $(e_i)_{i=1,\dots,n}$ is an orthonormal local frame of $TM$.

At a point $p\in M$ where $f$ is differentiable, introduce the following endomorphism of $E_p$:
\[
\tau_{v,w}:=\rmc(v)\rmc(w)\otimes f^*\big(\bar{\rmc}(d_pf(v))\bar{\rmc}(d_pf(w))\big)\quad\mbox{for }v,w\in T_pM.
\]
The next two lemmas are from \cite[Section 4]{Baer26hyp-rad}. We include the proofs for completeness.

\begin{lemma}\label{L:D squ equ}
If $\varphi\in\ker D_f$ and $U_f(\varphi)=(\varphi_\alpha)_\alpha$, then it holds almost everywhere that
\[
\sum_\alpha|D\varphi_\alpha\otimes\Psi_\alpha|^2=\frac{1}{4}\sum_{j,k}\langle\tau_{e_j,e_k}\varphi,\varphi\rangle.
\]
\end{lemma}

\begin{proof}
For $\varphi\in\ker D_f$, at a point where $f$ is differentiable, by \eqref{E:Killing spinor} and \eqref{E:Df and D},
\[
\begin{aligned}
& \sum_\alpha|D\varphi_\alpha\otimes\Psi_\alpha|^2 \\
& =\Big|\sum_{\alpha,i}\rmc(e_i)\varphi_\alpha\otimes (f^*\nabla^{\SigS_+^n})_{e_i}\Psi_\alpha\Big|^2 \\
& =\frac{1}{4}\Big|\sum_{\alpha,i}\rmc(e_i)\varphi_\alpha\otimes f^*\big(\bar{\rmc}(df(e_i))\big)\Psi_\alpha\Big|^2 \\
& =\frac{1}{4}\sum_{\beta,\gamma,j,k}\Big\langle\rmc(e_j)\varphi_\beta\otimes f^*\big(\bar{\rmc}(df(e_j))\big)\Psi_\beta,\rmc(e_k)\varphi_\gamma\otimes f^*\big(\bar{\rmc}(df(e_k))\big)\Psi_\gamma\Big\rangle \\
& =\frac{1}{4}\sum_{\beta,\gamma,j,k}\Big\langle\rmc(e_j)\rmc(e_k)\varphi_\beta\otimes f^*\big(\bar{\rmc}(df(e_j))\bar{\rmc}(df(e_k))\big)\Psi_\beta,\varphi_\gamma\otimes\Psi_\gamma\Big\rangle \\
& =\frac{1}{4}\sum_{j,k}\langle\tau_{e_j,e_k}\varphi,\varphi\rangle.
\end{aligned}
\]
\end{proof}

\begin{lemma}\label{L:D squ est}
Let $\varphi\in\ker D_f$ with $U_f(\varphi)=(\varphi_\alpha)_\alpha$, then
\begin{equation}\label{E:D squ est}
\sum_\alpha|D\varphi_\alpha|^2\le\frac{1}{4}|df|_1^2\sum_\alpha|\varphi_\alpha|^2
\end{equation}
holds almost everywhere on $M$. Here $|\cdot|_1$ denotes the trace norm.

Moreover, if $f$ is differentiable at $p$, if $d_pf$ is a linear isometry, and if equality in \eqref{E:D squ est} is achieved at $p$, then
\begin{equation}\label{E:D squ est eq}
\tau_{v,w}\varphi(p)=\varphi(p)
\end{equation}
for all orthonormal $v,w\in T_pM$.
\end{lemma}

\begin{proof}
Fix a point $p\in M$ where $f$ is differentiable and choose orthonormal bases $(e_i)_{i=1,\dots,n}$ of $T_pM$ and $(\oe_i)_{i=1,\dots,n}$ of $T_{f(p)}\rmS^n_+$ such that $d_pf(e_i)=\mu_i\oe_i$, where $0\le\mu_1\le\dots\le\mu_n$ are the singular values of $d_pf$.

Since $(\Psi_\alpha)_\alpha$ are pointwise orthonormal, by Lemma \ref{L:D squ equ},
\[
\sum_\alpha|D\varphi_\alpha|^2=\sum_\alpha|D\varphi_\alpha\otimes\Psi_\alpha|^2=\frac{1}{4}\sum_{j,k}\langle\tau_{e_j,e_k}\varphi,\varphi\rangle.
\]
In our choice of orthonormal bases,
\[
\tau_{e_j,e_k}=\mu_j\mu_k\rmc(e_j)\rmc(e_k)\otimes f^*\big(\bar{\rmc}(\oe_j)\bar{\rmc}(\oe_k)\big).
\]
Note that
\[
\big\langle\big(\rmc(e_j)\rmc(e_k)\otimes f^*\big(\bar{\rmc}(\oe_j)\bar{\rmc}(\oe_k)\big)\big)\varphi,\varphi\big\rangle\le|\varphi|^2
\]
with equality if and only if $j=k$ or
\[
\big(\rmc(e_j)\rmc(e_k)\otimes f^*\big(\bar{\rmc}(\oe_j)\bar{\rmc}(\oe_k)\big)\big)\varphi=\varphi\quad\mbox{for }j\ne k.
\]
Therefore,
\begin{equation}\label{E:D squ est 2}
\sum_\alpha|D\varphi_\alpha|^2\le\frac{1}{4}\sum_{j,k}\mu_j\mu_k|\varphi|^2=\frac{1}{4}|df|_1^2\sum_\alpha|\varphi_\alpha|^2.
\end{equation}

Now assume $d_pf$ is a linear isometry. Then $\mu_1=\cdots=\mu_n=1$ at $p$, and the orthonormal basis $(e_i)_{i=1,\dots,n}$ can be chosen arbitrarily. So the equality in \eqref{E:D squ est 2} yields \eqref{E:D squ est eq}.
\end{proof}

%%%%%%%%%%%%%%%%%%%%%%
\subsection{Proof of Theorem \ref{T:Dirac eig HHS}}\label{SS:pf thm}

Let $f:M\to\rmS_+^n$ be a boundary-preserving Lipschitz map of nonzero degree with $\Lip(f)=1/\Rad_{\rmS_+^n}(M)$. From Corollary \ref{C:ker D_f chi}, one can find a nonzero $\varphi\in\ker D_f$ such that $U_f(\varphi)=(\varphi_\alpha)_\alpha$ satisfies the positive or negative chirality boundary condition. By Remark \ref{R:chi bc}, either case can be used in the square Dirac eigenvalue estimate. Integrating \eqref{E:D squ est} over $M$ and using the fact that $|df|_1\le n/\Rad_{\rmS_+^n}(M)$, we obtain
\begin{equation}\label{E:Dirac eig-pf}
\lambda_1(D)^2\sum_\alpha\|\varphi_\alpha\|^2_{L^2(M)}\le\sum_\alpha\|D\varphi_\alpha\|^2_{L^2(M)}\le\frac{n^2}{4\Rad_{\rmS_+^n}(M)^2}\sum_\alpha\|\varphi_\alpha\|^2_{L^2(M)}.
\end{equation}
So \eqref{E:Dirac eig HHS} is proved.

For the ``if'' part of the equality statement, we know from \eqref{E:HHS R-hemisp} that $\Rad_{\rmS_+^n}(M)=R$ for $M=\rmS_+^n(R)$. On the other hand, by Theorem \ref{T:Dirac eig lower},
\[
\lambda_1(D)^2\ge\frac{n}{4(n-1)}\cdot\frac{n(n-1)}{R^2}=\frac{n^2}{4R^2}.
\]
Combined with \eqref{E:Dirac eig HHS}, equality holds in \eqref{E:Dirac eig HHS}.

The proof of the ``only if'' part is mostly in line with B\"ar's strategy in \cite[Sections 5.5--5.7]{Baer26hyp-rad}. We give a sketch below showing the necessary modifications adapted to our setting.

By rescaling the metric, we assume that $\Rad_{\rmS_+^n}(M)=1$. For the nonzero section $\varphi$ given earlier, equalities in \eqref{E:D squ est} and \eqref{E:Dirac eig-pf} imply the following facts:
\begin{enumerate}
\item \label{It:equality facts-1} For every $\alpha$, $\varphi_\alpha$ belongs to the $n^2/4$-eigenspace of $D_\calG^2$ or $D_{-\calG}^2$ (cf. Remark \ref{R:Dirac square eig}), so $\varphi_\alpha$ is smooth;
\item \label{It:equality facts-2} The zero locus $\calL$ of $U_f(\varphi)=(\varphi_\alpha)_\alpha$ has Hausdorff codimension at least 2 in $M$;
\item \label{It:equality facts-3} $df$ is a linear isometry almost everywhere;
\item \label{It:equality facts-4} $\tau_{v,w}\varphi(p)=\varphi(p)$ for almost all $p\in M$ and any orthonormal $v,w\in T_pM$.
\end{enumerate}
To see \ref{It:equality facts-1}, equalities in \eqref{E:Dirac eig-pf} imply $\|D\varphi_\alpha\|^2_{L^2(M)}=\lambda_1(D)^2\|\varphi_\alpha\|^2_{L^2(M)}$ with $\lambda_1(D)^2=n^2/4$, which means that each $\varphi_\alpha$ lies in the span of the eigenspinors of $D_\calG$ or $D_{-\calG}$ with eigenvalues $\pm\lambda_1(D)=\pm n/2$. Thus $\varphi_\alpha$ is smooth. It also follows that $\varphi_\alpha$ belongs to $\dom(D_\calG^2)$ or $\dom(D_{-\calG}^2)$, and $D^2\varphi_\alpha=\frac{n^2}{4}\varphi_\alpha$. To see \ref{It:equality facts-2}, put
\[
\Phi=\left(
\begin{matrix}
(\varphi_\alpha)_\alpha \\
\frac{2}{n}D(\varphi_\alpha)_\alpha
\end{matrix}\right),\qquad
\calD=\left(
\begin{matrix}
0 & D \\
D & 0
\end{matrix}\right).
\]
Then $\calD\Phi=\frac{n}{2}\Phi$. From \eqref{E:D squ est}, \eqref{E:Dirac eig-pf} and the fact that $|df|_1\le n$, we have the pointwise equality $\sum_\alpha|D\varphi_\alpha|^2=\frac{n^2}{4}\sum_\alpha|\varphi_\alpha|^2$. Hence $(\varphi_\alpha)_\alpha$ and $\Phi$ have the same zero locus. Applying the main theorem of \cite{Baer97nodal} (which holds for compact manifolds with boundary) to $\Phi$ gives the assertion. Equalities in \eqref{E:Dirac eig-pf} also imply $|df|_1=n$ almost everywhere on $M\setminus\calL$. Together with the fact that $\Lip(f)=1$, we get $\mu_1=\cdots=\mu_n=1$ almost everywhere. This shows \ref{It:equality facts-3}. \ref{It:equality facts-4} then follows from the equality condition of \eqref{E:D squ est}.

Note that $M$ is connected and $\rmS_+^n$ is simply connected, to prove that $f$ is a Riemannian isometry, we only need to show that $f$ is a smooth Riemannian covering.
By \cite[Theorem 1]{KMS19arma} (which again holds for compact manifolds with smooth boundary), it suffices to show that $df$ preserves (or reverses) the orientation almost everywhere.

\begin{lemma}\label{L:vol form action}
The above nonzero spinor $\varphi\in\ker D_f$ can always be chosen such that
\begin{equation}\label{E:vol form action}
(\rmc(\omega_\CC)\otimes f^*\bar{\rmc}(\bar{\omega}_\CC))\varphi=\epsilon\varphi,
\end{equation}
where $\epsilon=\pm1$ is constant.
\end{lemma}

\begin{proof}
When $n$ is even, this is immediate from Proposition \ref{P:ind D_f} as $\varphi$ is a section of either $E^+$ or $E^-$. When $n$ is odd, note that $\rmc(\omega_\CC)\otimes f^*\bar{\rmc}(\bar{\omega}_\CC)$ commutes with $D_f$, $G\otimes f^*\oG$, and $\tilcalG$. Hence one may replace $\varphi$ by a nonzero one of
\[
\frac{1}{2}\big(1\pm\rmc(\omega_\CC)\otimes f^*\bar{\rmc}(\bar{\omega}_\CC)\big)\varphi
\]
such that it is still in $\ker D_f$, still lying in $E^+$ or $E^-$, and satisfies the same boundary condition as $\varphi$ does.
\end{proof}

With this lemma, choosing a positively-oriented orthonormal basis $(e_i)_{i=1,\dots,n}$ of $T_pM$, with $f$ differentiable at $p$, we can write
\begin{equation}\label{E:df orient}
\begin{aligned}
&\quad (-1)^{[\frac{n+1}{2}]}(\rmc(\omega_\CC)\otimes f^*\bar{\rmc}(\bar{\omega}_\CC))\varphi(p) \\
& =\sigma(p)\big(\rmc(e_1)\cdots\rmc(e_n)\otimes f^*(\bar{\rmc}(df(e_1))\cdots\bar{\rmc}(df(e_n)))\big)\varphi(p).
\end{aligned}
\end{equation}
By \ref{It:equality facts-3}, $\sigma(p)=\pm1$ if $d_pf$ is orientation preserving/reversing. It then remains to show that $\sigma$ is constant almost everywhere.

If $n$ is even, then the right-hand side of \eqref{E:df orient} is exactly
\[
\sigma(p)\cdot\tau_{e_1,e_2}\circ\cdots\circ\tau_{e_{n-1},e_n}\varphi(p)=\sigma(p)\varphi(p)
\]
by \ref{It:equality facts-4}. Combined with \eqref{E:vol form action}, $\sigma=(-1)^{[\frac{n+1}{2}]}\epsilon$ almost everywhere.

If $n$ is odd, the action of $\rmc(e_1)\otimes f^*\bar{\rmc}(df(e_1))$ is left after a simplification using  \ref{It:equality facts-4}. Since $e_1$ is chosen arbitrarily, we have
\[
(\rmc(v)\otimes f^*\bar{\rmc}(df(v)))\varphi(p)=(-1)^{[\frac{n+1}{2}]}\epsilon\sigma(p)\varphi(p)
\]
for any unit $v\in T_pM$. By Lemma \ref{L:D_f vs. D} and \cite[Lemma 3]{Baer26hyp-rad} (a pointwise formula, so independent of the boundary presence), a computation yields for each $\alpha$,
\[
D\varphi_\alpha=(-1)^{[\frac{n+1}{2}]}\frac{n\epsilon}{2}\sigma(p)\varphi_\alpha.
\]
By \ref{It:equality facts-1}, $\sigma$ is smooth and hence locally constant on $M\setminus\calL$. It follows from \ref{It:equality facts-2} that $M\setminus\calL$ is connected. Consequently, $\sigma$ is constant almost everywhere. This finishes the proof of the ``only if'' part, hence the whole proof of Theorem \ref{T:Dirac eig HHS}.

%%%%%%%%%%%%%%%%%%%%%%
\subsection{A refined hyper-hemispherical radius}\label{SS:ref HHS}

Similar to \cite[Section 7.2]{Baer26hyp-rad}, from the proof of Theorem \ref{T:Dirac eig HHS}, our upper Dirac eigenvalue estimate in fact holds in the more precise form of \eqref{E:Dirac eig HHS-tr}. We repeat it below:
\begin{equation}\label{E:Dirac eig HHS-tr-2}
\lambda_1(D)^2\le\frac{n^2}{4\Rad_{\rmS_+^n}^{\rm tr}(M)^2},
\end{equation}
where the \emph{trace-refined hyper-hemispherical radius} $\Rad_{\rmS_+^n}^{\rm tr}(M)$ is defined to be the supremum of all $R>0$ such that there exists a boundary-preserving Lipschitz map $f:M\to\rmS_+^n$ of nonzero degree such that the trace norm $|df|_1\le n/R$ almost everywhere. It is clear that
\[
\Rad_{\rmS_+^n}^{\rm tr}(M)\ge\Rad_{\rmS_+^n}(M).
\]

The following example indicates that the above inequality can be strict, so \eqref{E:Dirac eig HHS-tr-2} is indeed a stronger estimate.

\begin{example}\label{Eg:sph cap-tr}
Consider the spherical cap
\[
M=\{x\in\rmS^n:\dist(x,p_0)\le\rho\}
\]
as in Example \ref{Eg:sph cap} again. Then for the map $f:M\to\rmS_+^n$ defined in \eqref{E:sph cap map}, we have
\[
|df|_1\le
\begin{cases}
\frac{n\pi}{2\rho}, & 0<\rho\le\frac{\pi}{2}, \\
\frac{\pi}{2\rho}+\frac{n-1}{\sin\rho}, & \frac{\pi}{2}<\rho<\pi,
\end{cases}
\]
so that
\[
\Rad_{\rmS_+^n}^{\rm tr}(M)\ge\begin{cases}
\frac{2\rho}{\pi}, & 0<\rho\le\frac{\pi}{2}, \\
\frac{2n\rho\sin\rho}{\pi\sin\rho+2(n-1)\rho}, & \frac{\pi}{2}<\rho<\pi.
\end{cases}
\]
Compared with \eqref{E:sph cap rad}, we see that $\Rad_{\rmS_+^n}^{\rm tr}(M)>\Rad_{\rmS_+^n}(M)$ for $\pi/2<\rho<\pi$.
\end{example}

However, it should be pointed out that the equality case of \eqref{E:Dirac eig HHS-tr-2} does not necessarily lead to a hemisphere, as the following proposition illustrates.

\begin{proposition}\label{P:non-hemisp eq}
Let $\Omega\subset\RR^{n+1}$ ($n\ge2$) be a smooth connected closed hypersurface with positive principal curvatures which is reflection symmetric about the hyperplane $P=\{x_{n+1}=0\}$. Set
\[
M=\Omega\cap\{x_{n+1}\ge0\},\qquad g=H^2g_0|_M,
\]
where $g_0$ is the induced metric on $\Omega$ and $H$ denotes the normalized (positive) mean curvature of $\Omega$ (as the boundary of a Euclidean domain). Then
\begin{equation}\label{E:non-hemisp eq}
\Rad_{\rmS_+^n}^{\rm tr}(M,g)=1,\qquad \lambda_1(D_g)^2=\frac{n^2}{4},
\end{equation}
where $D_g$ denotes the spin Dirac operator on $M$ with respect to the metric $g$, and the eigenvalue is under a chirality boundary condition.

Moreover, if $\Omega$ is not a Euclidean sphere, then $(M,g)$ is not a round hemisphere.
\end{proposition}

\begin{proof}
Consider the outward Gauss map from $\Omega$ to $\rmS^n$. By the symmetry and strict convexity of $\Omega$, this map restricts to a boundary-preserving diffeomorphism $f:M\to\rmS_+^n$. With respect to $g$, the singular values of $df$ are $\kappa_i/H$, with $\kappa_1,\dots,\kappa_n$ denoting the principal curvatures. Consequently,
\[
|df|_1=\frac{\kappa_1+\cdots+\kappa_n}{H}=n.
\]
Hence $\Rad_{\rmS_+^n}^{\rm tr}(M,g)\ge1$.

In view of \eqref{E:Dirac eig HHS-tr-2}, to prove \eqref{E:non-hemisp eq}, it remains to show that $\lambda_1(D_g)^2\ge n^2/4$. In our setting, we can identify the graded spinor bundle $\Sigma\Omega$ over $\Omega$ with the restriction of the ambient spinor bundle (compare Remark \ref{R:chi op}). Let $N$ be the inward unit normal vector field on $\Omega$. With respect to the induced metric $g_0$, the Clifford multiplication and spin connection are defined by
\begin{equation}\label{E:CM & conn from ambient}
\rmc_{g_0}(X)=\gamma(X)\gamma(N),\qquad \nabla^{g_0}_X=\bar{\nabla}_X-\frac{1}{2}\gamma(A(X))\gamma(N),
\end{equation}
where $\gamma(\cdot)$ (resp. $\bar{\nabla}$) denotes the ambient Clifford multiplication (resp. connection), and $A$ is the shape operator of $\Omega$ (whose trace is $nH$). With respect to the conformal metric $H^2g_0$ (which we still denote by $g$ on the whole $\Omega$), they are defined by
\[
\rmc_g(X)=H\gamma(X)\gamma(N),\qquad \nabla^g_X=\nabla^{g_0}_X-\frac{1}{2H}\gamma(X)\gamma(\grad_{g_0}H)-\frac{1}{2H}g_0(X,\grad_{g_0}H).
\]
Let $\hatD_{g_0}$ and $\hatD_g$ denote the corresponding spin Dirac operators on $\Omega$. They are related by
\begin{equation}\label{E:conf Dirac}
\hatD_g=H^{-\frac{n+1}{2}}\hatD_{g_0}H^{\frac{n-1}{2}}.
\end{equation}
See \cite[Section 4]{HijMon14}. For both metrics, one can check that a chirality operator is given by $G=\sqrt{-1}\gamma(N)$.

Let $e_{n+1}$ be the unit Euclidean vector in the direction of $x_{n+1}$. By the reflection symmetry of $\Omega$, the inward unit normal vector field along $\pM$ is $\nu_g=H^{-1}e_{n+1}$ under the metric $g$. Set
\[
\chi:=\rmc_g(\nu_g)G=-\sqrt{-1}\gamma(e_{n+1}).
\]
Let $\rho$ be the reflection map on $\RR^{n+1}$ about the plane $P$. Define a reflection map $T$ on $\Sigma\Omega$ by
\[
T(\varphi)(p):=\chi\varphi(\rho(p)).
\]
Note that for any ambient vector $v$, the Clifford relations give $\chi\gamma(v)\chi^{-1}=-\gamma(\rho(v))$.
By symmetry,
\[
N(\rho(p))=\rho(N)(p),\qquad A_{\rho(p)}(\rho(X))=\rho(A_p(X)).
\]
Put these together with \eqref{E:CM & conn from ambient}, we have $T^{-1}\hatD_{g_0}T=\hatD_{g_0}$.
Then by \eqref{E:conf Dirac} and the reflection symmetry of $H$, it holds that
\begin{equation}\label{E:ref inv Dirac}
T^{-1}\hatD_gT=\hatD_g.
\end{equation}
Now suppose $\varphi$ is a nonzero eigenspinor of $D_g$ under the chirality boundary condition associated with $G$. Extend it to a spinor field on the whole $\Omega$ by
\[
\hat{\varphi}(p):=
\begin{cases}
\varphi(p), & p\in M, \\
T(\varphi)(p), & p\in\rho(M).
\end{cases}
\]
The boundary condition ensures that $\chi\varphi=\varphi$ on $\pM$. So this is an $H^1$-extension. It follows from \eqref{E:ref inv Dirac} that $\hat{\varphi}$ is an eigenspinor of $\hatD_g$ in the weak sense. Elliptic regularity makes this weak eigenspinor smooth. In this situation, we can apply an estimate of Hijazi--Montiel \cite[Theorem 1, Proposition 9]{HijMon14} to conclude
\[
\lambda_1(\hatD_g)^2\ge\frac{n^2}{4},
\]
which implies the same lower bound for $\lambda_1(D_g)^2$. The first part of the proposition is proved.

To prove the second part, note that $H^n\ge\kappa_1\cdots\kappa_n$, with strict inequality at every point that is not umbilic. If $\Omega$ is not a Euclidean sphere, then the inequality is strict on an open subset of $M$. Hence
\[
\begin{aligned}
\Vol(M,g)=\int_MH^ndV_{g_0}>\int_M\kappa_1\cdots\kappa_ndV_{g_0}=\Vol(\rmS_+^n).
\end{aligned}
\]
This means that $(M,g)$ is not a round hemisphere, because a round hemisphere satisfying \eqref{E:non-hemisp eq} must be $\rmS_+^n$.
\end{proof}

\begin{remark}\label{R:non-hemisp eq}
In view of the rigidity part of Theorem \ref{T:Dirac eig HHS}, one has that $\Rad_{\rmS_+^n}^{\rm tr}(M,g)>\Rad_{\rmS_+^n}(M,g)$ for $(M,g)$ as in Proposition \ref{P:non-hemisp eq}.
\end{remark}

\begin{remark}\label{R:non-sp eq}
When the hypersurface $\Omega$ given in Proposition \ref{P:non-hemisp eq} (with no symmetry needed) is not a Euclidean sphere, $(\Omega,H^2g_0)$ also serves as a non-round equality example in the refined upper Dirac eigenvalue estimate of B\"ar \cite[Section 7.2]{Baer26hyp-rad}.
\end{remark}

%%%%%%%%%%%%%%%%%%%%%%
\subsection{Proof of Corollary \ref{C:hemisp rigidity}}\label{SS:pf cor}

Under the hypothesis of Corollary \ref{C:hemisp rigidity}, $\Rad_{\rmS_+^n}^{\rm tr}(M)\ge1$ by definition. Together with \eqref{E:eig lower} and \eqref{E:Dirac eig HHS-tr-2}, we have
\[
\frac{n^2}{4}\le\frac{n}{4(n-1)}\inf_M\scal\le\lambda_1(D)^2\le\frac{n^2}{4\Rad_{\rmS_+^n}^{\rm tr}(M)^2}\le\frac{n^2}{4}.
\]
Thus all inequalities are equalities. In particular, by Theorem \ref{T:Dirac eig lower}, $M$ is isometric to $\rmS_+^n$. It remains to show that $f$ is an isometry.

Like before, let $\mu_1,\dots,\mu_n$ be the singular values of $df$ at points where $f$ is differentiable. Then
\begin{equation}\label{E:AM-GM}
\mu_1\cdots\mu_n\le\bigg(\frac{\mu_1+\cdots+\mu_n}{n}\bigg)^n=\bigg(\frac{|df|_1}{n}\bigg)^n\le1.
\end{equation}
Without loss of generality, we may assume that $\deg(f)>0$. Using a similar argument as in \cite[Remark 4]{Baer26hyp-rad},
\begin{equation}\label{E:vol comparison}
\begin{aligned}
\Vol(M)& =\int_MdV_M\ge\int_M\mu_1\cdots\mu_ndV_M\ge\int_Mf^*dV_{\rmS_+^n} \\
& =\deg(f)\Vol(\rmS_+^n)\ge\Vol(\rmS_+^n)=\Vol(M).
\end{aligned}
\end{equation}
So all the inequalities in \eqref{E:AM-GM} and \eqref{E:vol comparison} are equalities. This means that $\mu_1=\cdots=\mu_n=1$, $df$ preserves the orientation almost everywhere, and $\deg(f)=1$. Again by \cite[Theorem 1]{KMS19arma}, $f$ is an isometry.

\begin{remark}\label{R:sp rigidity}
When $M$ is closed and $f:M\to\rmS^n$ is a Lipschitz map of nonzero degree such that $|df|_1\le n$ almost everywhere, Corollary \ref{C:hemisp rigidity} still holds. For the proof, the estimate of Hijazi--Montiel--Rold\'an is replaced by Friedrich's estimate. But the equality case in the latter estimate is not necessarily a sphere. Nevertheless, the equality characterization gives $\Ric=(n-1)g$ on $(M,g)$, which implies $\Vol(M)\le\Vol(\rmS^n)$ using the Bishop--Gromov volume comparison theorem. The rest of the argument is the same.
\end{remark}

%%%%%%%%%%%%%%%%%%%%%%
%%%%%%%%%%%%%%%%%%%%%%
\section{Refined estimates}\label{S:ref est}

In this section, we use the trace-refined hyper-hemispherical radius to provide concrete upper Dirac eigenvalue estimates for certain spaces.

%%%%%%%%%%%%%%%%%%%%%%
\subsection{Estimates on geodesic annuli}\label{SS:annuli est}

One sees from Example \ref{Eg:sph cap-tr} that \eqref{E:Dirac eig HHS-tr-2} serves as an improvement in the mean concave boundary scenario. In this subsection, we make a more general examination for geodesic annuli.

Let $X$ be a Riemannian $n$-manifold. A \emph{geodesic annulus} in $X$ is defined to be
\[
M=\overline{B(p,R)}\setminus B(p,r),
\]
where $B(p,R)$ (or $B(p,r)$) denotes the open geodesic ball centered at $p\in X$ with radius $R$ (or $r$). Here $R>r>0$ and $R$ is less than the injectivity radius of $X$ at $p$.

For simplicity, we will focus on the situation that $X$ has nonpositive sectional curvature. Let $\p_-M$ denote the inner sphere $S(p,r)$ and $\p_+M$ denote the outer sphere $S(p,R)$. Write points in $M$ in geodesic polar coordinates as
\[
x=\exp_p(tu),\qquad r\le t\le R,\ |u|=1.
\]
Choose a number $s\in(r,R)$ and divide the annulus $M$ into
\[
M_-:=\{x\in M:r\le t\le s\},\qquad M_+:=\{x\in M:s\le t\le R\}.
\]
We construct a map $f:M\to\rmS_+^n$ such that $M_-$ is collapsed angularly to a single meridian and $M_+$ is mapped radially to $\rmS_+^n$ in the polar way.

Let $q\in\rmS_+^n$ be the north pole and choose an orientation-preserving linear isometry
\[
A:T_pX\to T_q\rmS_+^n.
\]
Put
\[
a_-:=\frac{\pi}{2(s-r)},\qquad a_+:=\frac{\pi}{2(R-s)}.
\]
Define
\[
f(t,u):=
\begin{cases}
\exp_q(a_-(s-t)v_0)\quad & \mbox{on }M_-, \\
\exp_q(a_+(t-s)A(u))& \mbox{on }M_+,
\end{cases}
\]
where $v_0$ is a fixed unit vector in $T_q\rmS_+^n$. One can check that $f$ is a boundary-preserving Lipschitz map. Also, $\p_+M$ is mapped onto $\p\rmS_+^n$ with degree 1 while $\p_-M$ is mapped to a single point in $\p\rmS_+^n$ with degree 0. Thus $\deg(f)=1$.

To estimate the trace norm of $df$, set $f_\pm:=f|_{M_\pm}$. The radial Lipschitz constant of $f_\pm$ is $a_\pm$. For an angular Jacobi field $J$ along a radial geodesic $\gamma(t)=\exp_p(tu)$ in $M_+$ with $J(0)=0$, since $X$ has nonpositive curvature, Rauch comparison theorem gives $|J(t)|\ge t|J'(0)|$. On $\rmS_+^n$, at spherical radius $\theta(t)=a_+(t-s)$, the corresponding angular Jacobi field has norm $\sin\theta(t)|J'(0)|$. So each angular singular value of $df_+$ is bounded by
\[
Q(s):=\sup_{s\le t\le R}\frac{\sin(a_+(t-s))}{t}\le a_+.
\]
Put them together, we have
\[
|df|_1\le\max\{a_-,a_++(n-1)Q(s)\}\le\max\bigg\{\frac{\pi}{2(s-r)},\frac{n\pi}{2(R-s)}\bigg\}\quad\mbox{a.e.}
\]
As $s$ varies between $r$ and $R$, the minimum of the right-hand side is achieved when $\frac{1}{s-r}=\frac{n}{R-s}$. In this case
\[
|df|_1\le\frac{(n+1)\pi}{2(R-r)}\quad\mbox{a.e.}
\]
Therefore, we obtain
\[
\Rad_{\rmS_+^n}^{\rm tr}(M)\ge\frac{2n}{n+1}\frac{R-r}{\pi}.
\]
Using \eqref{E:Dirac eig HHS-tr-2}, we prove the following Dirac eigenvalue estimate.

\begin{proposition}\label{P:Dirac eig annuli}
Let $M=\overline{B(p,R)}\setminus B(p,r)$ be a geodesic annulus in a spin Riemannian $n$-manifold $X$ of nonpositive sectional curvature as above. Then under a chirality boundary condition,
\begin{equation}\label{E:Dirac eig annuli}
\lambda_1(D)^2\le\frac{(n+1)^2\pi^2}{16(R-r)^2}.
\end{equation}
\end{proposition}

\begin{remark}\label{R:Dirac eig annuli}
If $X$ is a Cartan--Hadamard manifold, then the injectivity radius at every point is infinite. In this case \eqref{E:Dirac eig annuli} holds for any $p\in X$ and any outer radius $R$.
\end{remark}

%%%%%%%%%%%%%%%%%%%%%%
\subsection{Estimates on warped product spaces}\label{SS:warp-prod est}

In this subsection, we find a simple lower bound for the trace-refined hyper-hemispherical radius of a warped product space. Consider $M=[0,1]\times N$ with a warped product metric
\begin{equation}\label{E:warp-prod}
g=dt^2+\phi(t)^2h,
\end{equation}
where $h$ is a metric on a closed manifold $N$ and $\phi(t)>0$.

Set $R_N:=\Rad_{\rmS^{n-1}}^{\rm tr}(N)$ to be the trace-refined hyperspherical radius of $N$ in the sense of \cite[Section 7.2]{Baer26hyp-rad}. Then for any small $\varepsilon>0$, there exists a Lipschitz map $f:N\to\rmS^{n-1}$ of non-zero degree such that
\[
|df|_1\le\frac{n-1}{R_N-\varepsilon}\quad\mbox{a.e.}
\]
As in last subsection, choose a number $s\in(0,1)$ and define $F:M\to\rmS_+^n$ by
\[
F(t,x):=
\begin{cases}
(\frac{\pi}{2s}(s-t),\omega_0),\qquad & 0\le t\le s, \\
(\frac{\pi}{2(1-s)}(t-s),f(x)),\ & s\le t\le 1,
\end{cases}
\]
where $\omega_0\in\p\rmS_+^n$ is fixed. This is clearly a boundary-preserving Lipschitz map of non-zero degree.

Again as in last subsection, we have
\begin{equation}\label{E:trace norm warp-prod}
|dF|_1\le\max\bigg\{\frac{\pi}{2s},\frac{\pi}{2(1-s)}+\sup_{s\le t\le 1}\frac{n-1}{R_N-\varepsilon}\frac{\sin(\frac{\pi}{2(1-s)}(t-s))}{\phi(t)}\bigg\}\quad\mbox{a.e.}
\end{equation}
To get a more explicit (but fairly rough) estimate, put
\[
\phi_*:=\min_{0\le t\le 1}\phi(t)\quad\mbox{and}\quad A_\varepsilon:=\frac{n-1}{(R_N-\varepsilon)\phi_*}.
\]
Then
\[
|dF|_1\le\max\bigg\{\frac{\pi}{2s},\frac{\pi}{2(1-s)}+A_\varepsilon\bigg\}\quad\mbox{a.e.}
\]
When $\frac{\pi}{2s}=\frac{\pi}{2(1-s)}+A_\varepsilon$, the right-hand side is minimized and one has
\[
|dF|_1\le\frac{1}{2}(A_\varepsilon+\pi+\sqrt{A_\varepsilon^2+\pi^2})\quad\mbox{a.e.}
\]
Now letting $\varepsilon\searrow 0$ yields
\[
\Rad_{\rmS_+^n}^{\rm tr}(M)\ge\frac{2n}{A+\pi+\sqrt{A^2+\pi^2}},
\]
where $A:=\frac{n-1}{R_N\phi_*}$. This proves the following.

\begin{proposition}\label{P:Dirac eig warp-prod}
Let $N$ be a closed spin Riemannian $(n-1)$-manifold and $M=[0,1]\times N$ with the warped product metric \eqref{E:warp-prod} as above. Then under a chirality boundary condition,
\begin{equation}\label{E:Dirac eig warp-prod}
\lambda_1(D)^2\le\frac{(A+\pi+\sqrt{A^2+\pi^2})^2}{16}.
\end{equation}
\end{proposition}

\begin{remark}\label{R:Dirac eig warp-prod}
In the case that the warped-product space is collapsed at one end (e.g. the warped metric extends to a smooth metric after the end $\{0\}\times N$ is collapsed to a point), one can choose $s=0$ and \eqref{E:trace norm warp-prod} is reduced to
\[
|dF|_1\le\frac{\pi}{2}+\frac{n-1}{R_N-\varepsilon}\cdot\sup_{0<t\le 1}\frac{\sin\frac{\pi t}{2}}{\phi(t)}.
\]
The trace-refined hyper-hemispherical radius is then estimated by letting $\varepsilon$ approach 0. This can be seen as a generalization of Example \ref{Eg:sph cap-tr}.
\end{remark}

\begin{remark}\label{R:Dirac eig annuli+warp-prod}
Combining \eqref{E:Dirac eig annuli} or \eqref{E:Dirac eig warp-prod} with \eqref{E:Dirac eig con Lap} will also lead to an upper estimate for the first eigenvalue of the conformal Laplacian of a geodesic annulus or a warped-product space in dimensions $n\ge3$.
\end{remark}

%In particular, note that $\Rad_{\rmS^{n-1}}^{\rm tr}(\rmS^{n-1})=1$. (In fact, $|d\id|_1=n-1$ and the volume bound gives the opposite inequality.) So
%\[
%\Rad_{\rmS_+^n}^{\rm tr}([0,1]\times\rmS^{n-1})\ge
%\]

%\begin{theorem}\label{T:hemisp}
%Let $M$
%\end{theorem}

%%%%%%%%%%%%%%%%%%%%%
%%%%%%%%%%%%%%%%%%%%%
\section*{Acknowledgements and Disclosure on AI usage}\label{S:ackn}

%%%%%%%%%%%%%%%%%%%%%%
%\subsection*{Acknowledgements}\label{SS:ackn}

The author thanks Professor Ruijun Wu for discussions.

%%%%%%%%%%%%%%%%%%%%%%
%\subsection*{Disclosure on AI usage}\label{SS:AI usage}

The principal ideas of the paper are proposed by the author. During the preparation of the manuscript, generative AI tools (ChatGPT, Deepseek) have been used in assisting with computations (in Section 6) and improving the exposition. In particular, Proposition \ref{P:non-hemisp eq} and the formulation of Corollary \ref{C:hemisp rigidity} in terms of the trace norm of $df$ (instead of the Lipschitz constant of $f$) are suggested by AI. The manuscript is written by the author, who takes full responsibility for the content and mathematical arguments of the paper.

%%%%%%%%%%%------------
%%%%%%%%%%%%%%%%%%%%%%
%%%%%%%%%%%%%%%%%%%%%%
%---------------------------------------------------
\bibliographystyle{amsplain}
%\bibliography{bibdesk}

% \bib, bibdiv, biblist are defined by the amsrefs package.

\end{document}